\documentclass[a4paper,11pt]{article}

\usepackage[margin=1.2in]{geometry}
\newcommand\blfootnote[1]{
  \begingroup
  \renewcommand\thefootnote{}\footnote{#1}
  \addtocounter{footnote}{-1}
  \endgroup
}

\usepackage{cleveref}
\usepackage{theorem}
\usepackage{amsmath}
\usepackage{amssymb}
\usepackage{graphicx}
\usepackage{enumerate}
\usepackage[numbers,sort&compress]{natbib}

\usepackage{algorithm}
\usepackage{algorithmic}

\newcommand{\menge}[2]{\{{#1} ~|~ {#2}\}} 
\newcommand{\emp}{\ensuremath{{\varnothing}}}
\newcommand{\scal}[2]{\left\langle{#1}\mid {#2} \right\rangle} 
\newcommand{\HH}{\ensuremath{\mathcal H}}
\newcommand{\HHH}{\ensuremath{\boldsymbol{\mathcal H}}}
\newcommand{\RR}{\ensuremath{\mathbb R}}
\newcommand{\NN}{\ensuremath{\mathbb N}}
\newcommand{\dom}{\ensuremath{\operatorname{dom}}}
\newcommand{\prox}{\ensuremath{\operatorname{prox}}}
\newcommand{\ran}{\ensuremath{\operatorname{ran}}}
\newcommand{\zer}{\ensuremath{\operatorname{zer}}}
\newcommand{\gra}{\ensuremath{\operatorname{gra}}}
\newcommand{\vv}{\ensuremath{\boldsymbol{v}}}
\newcommand{\xx}{\ensuremath{\boldsymbol{x}}}
\newcommand{\yy}{\ensuremath{\boldsymbol{y}}}
\newcommand{\ee}{\ensuremath{\boldsymbol{e}}}
\newcommand{\zz}{\ensuremath{\boldsymbol{z}}}
\newcommand{\uu}{\ensuremath{\boldsymbol{u}}}
\newcommand{\ww}{\ensuremath{\boldsymbol{w}}}
\newcommand{\BB}{\ensuremath{\boldsymbol{B}}}
\newcommand{\E}{\ensuremath{\mathbf{E}}}
\newcommand{\AAA}{\ensuremath{\boldsymbol{A}}}
\newcommand{\BBB}{\ensuremath{\boldsymbol{B}}}
\newcommand{\QQ}{\ensuremath{\boldsymbol{Q}}}
\newcommand{\Fix}{\ensuremath{\operatorname{Fix}}}
\newcommand{\Id}{\ensuremath{\operatorname{Id}}}
\newcommand{\weakly}{\ensuremath{\rightharpoonup}}

\newtheorem{theorem}{Theorem}
\newtheorem{lemma}{Lemma}
\theoremstyle{plain}{\theorembodyfont{\rmfamily}
\newtheorem{problem}{Problem}}
\theoremstyle{plain}{\theorembodyfont{\rmfamily}
\newtheorem{remark}{Remark}}
\theoremstyle{plain}{\theorembodyfont{\rmfamily}

\def\endproof{\hfill\vbox{\hrule height0.6pt\hbox{\vrule height1.3ex
width0.6pt\hskip0.8ex\vrule width0.6pt}\hrule height0.6pt}}

\begin{document}

\title{\sffamily\LARGE 
Inertial Three-Operator Splitting Method and Applications
}
\author{ Volkan Cevher ~~~~~ B$\grave{\text{\u{a}}}$ng C\^ong V\~u ~~~~~ Alp Yurtsever \\[5mm]
Laboratory for Information and Inference Systems (LIONS) \\[0.5em]
Ecole Polytechnique Federale de Lausanne (EPFL), Switzerland \\[0.5em]
volkan.cehver@epfl.ch; bang.vu@epfl.ch; alp.yurtsever@epfl.ch
}
\date{}
\maketitle

\blfootnote{This work is presented at SIAM Conference on Optimization (OP17) in Vancouver, British Columbia, Canada on the 23rd of May 2017 by B$\grave{\text{\u{a}}}$ng C\^ong V\~u.}

\begin{abstract}
We introduce an inertial variant of the forward-Douglas-Rachford splitting and analyze its convergence. We specify an instance of the proposed method to the three-composite convex minimization template. We provide practical guidance on the selection of the inertial parameter based on the adaptive starting idea. Finally, we illustrate the practical performance of our method in various machine learning applications. 
\end{abstract}

\section{Introduction}

Consider the following abstract problem based on monotone inclusion of the sum of three-operators: 

\begin{problem} 
\label{prob:monotone-inclusion} ({Three-operators sum problem}) \\
Let $\beta$ be a strictly positive number,
$(\HHH,\scal{\cdot}{\cdot})$ be a real Hilbert space, $\AAA\colon\HHH\to2^{\HHH}$ 
and $\BB\colon\HHH\to 2^{\HHH}$ be maximally monotone operators,
and $\QQ\colon\HHH\to{\HHH}$ be a $\beta$-cocoercive operator, i.e.,
 \begin{equation*}
\scal{\xx-\yy}{\QQ\xx-\QQ\yy}\geq \beta \scal{\QQ\xx -\QQ\yy}{\QQ\xx-\QQ\yy}, \qquad
\forall \xx,\yy\in\HHH
 \end{equation*}
Let $\boldsymbol{\mathcal{X}^\star}$ be the set of all points $\xx$ in $\HHH$ such that 
\begin{equation*}
\mathbf{0} \in\AAA\xx + \BB\xx+\QQ\xx.
\end{equation*}
The problem is to find a point $\xx^\star$ in $\boldsymbol{\mathcal{X}^\star}$.  

\noindent \textbf{Assumption.} 
We assume that $\boldsymbol{\mathcal{X}^\star}$ is not empty. 
\end{problem}

Problem~\ref{prob:monotone-inclusion} generalizes the common two-operator sum problem templates,  including the sum of two maximally monotone operators (with $\QQ = 0$), 
and the sum of a maximally monotone operator and a cocoercive operator (with $\BBB=0$). The former can be solved by using forward-backward splitting \cite{Combettes2005signal},  and the latter by Douglas-Rachford splitting.
Moreover, it also covers the problem template of the forward-Douglas-Rachford splitting \cite{Briceno2015fdr}, where $\BBB$ is assumed to be the normal operator of a closed vector subspace.  

The general template of Problem~\ref{prob:monotone-inclusion} is recently solved in the three operator splitting framework \cite{Davis2017threeop}.
In this paper, we introduce and investigate the convergence characteristics of an inertial forward-Douglas-Rachford splitting method for solving Problem~\ref{prob:monotone-inclusion}. 

Such operator splitting schemes for finding the set of zero points of maximally monotone operators has a large number of applications in machine learning, statistics, signal processing and computer science in disguise. In particular, the three-composite convex optimization directly fits into this inclusion framework: 

\begin{problem}({Three-composite convex minimization}) \label{prob:convex-optimization} \\
Let $f\colon\RR^d\to\left]-\infty,+\infty\right]$ and $g\colon\RR^d\to\left]-\infty,+\infty\right]$ be  proper lower semicontinuous convex functions, and
let $h\colon\RR^d\to\RR$ be a differentiable convex function with $L$-Lipschitz 
continuous gradient, i.e., $ \forall \xx, \yy\in\RR^d$:
$$\| \nabla h(\xx) - \nabla h(\yy)\| \leq L \| \xx - \yy \|.$$
\noindent Then, we call the following template as the three-composite convex minimization problem:
\begin{equation*}
\begin{aligned}
& \underset{\xx \in \mathbb{R}^d}{\text{minimize}} &  f(\xx) + g(\xx) + h(\xx).
\end{aligned}
\end{equation*}
\end{problem}

Problem~\ref{prob:convex-optimization} is a special instance of Problem~\ref{prob:monotone-inclusion} and it covers many classical convex optimization templates as a special case, including the classical composite (objective is the sum of a smooth and a nonsmooth functions) and the constrained convex minimization problems. 
These special instances of Problem~\ref{prob:convex-optimization} can be solved using the proximal  gradient methods. 

Clearly, Problem~\ref{prob:convex-optimization} can be solved with the classical proximal gradient methods using the $\prox$ operator  (cf. \Cref{sec:notation}) of the joint term $f+g$. 
In contrast, our method makes use of 
the $\prox$ operators of $f$ and $g$ separately, similar to the methods described in \cite{Davis2017threeop,Raguet2013}. 
Note that the computation of the joint $\prox$ is more expensive compared to the individual $\prox$ operators, which can be observed even in the simplest examples with $f$ and $g$ being indicator functions of two convex sets (cf. \Cref{sec:DNNproj}). 

Inertial methods in monotone inclusions are first proposed in \cite{Alvarez2001ipm,Alvarez2004ihp} for finding the set of zero points of a single maximally monotone operator. 
Inertial variants of forward-backward and Douglas-Rachford splitting are investigated  in \cite{Moudafi2003,Lorenz2015ifb} and \cite{Bot2015iner} respectively. 
Some other extensions and modifications of the aforementioned results can be found in \cite{Moudafi2004,Bot2016iner,rosasco2016stochastic,Bednarczuk2016,Mainge2007,Pesquet2012}. 

Inertial methods in monotone inclusions are closely related with the accelerated proximal gradient method and its variants in convex optimization theory \cite{Nesterov1983acc,Nesterov2005smooth,Guler1992,Villa2013, attouch2016,attouch2017,Beck2009fista,Chambolle2015,Bonettini2017,Aujol2015, Moudafi2004,Bot2016iner,rosasco2016stochastic,Bednarczuk2016,Mainge2007,Pesquet2012}. 

To our knowledge, our framework presents the first purely primal inertial splitting method for solving Problem~\ref{prob:monotone-inclusion} without further assumptions. It is based on a combination of the three operator splitting method \cite{Davis2017threeop} and the inertial forward-backward splitting \cite{Moudafi2003,Lorenz2015ifb}, and it recovers these two schemes as a special case. After we submitted this manuscript for review, a similar approach has appeared very recently in a concurrent work \cite{Cui2019}. 

The paper is organized as follows: 
\Cref{sec:notation} presents the notation and recalls some basic notions from monotone inclusions. 
Then, \Cref{s:algo} introduces the inertial forward-Douglas-Rachford splitting method and proves the weak convergence. 
\Cref{s:convex-optimization} describes the application of the proposed method to three-composite convex minimization template, and \Cref{sec:restart} introduces the heuristic adaptive restart scheme. 
Finally, \Cref{s:numer} presents the numerical experiments.

\section{Notation \& Preliminaries}
\label{sec:notation}

This section recalls the basic notions from the monotone inclusion theory, and presents the key lemmas to be used in the sequel.  

Let $\HHH$ be a real Hilbert space with the inner product $\scal{\cdot}{\cdot}$ and the associated norm $\|\cdot\|$. 
For definitions given below, suppose that  $\AAA \colon \HHH \to 2^{\HHH}$ is a set-valued operator, and 
$f \colon \RR^d \to \left[-\infty,+\infty\right]$ is a proper, lower semicontinuous convex function. 

\textbf{Weak and strong convergence.} 
The symbols $\weakly$ and $\to$ denote the weak and strong convergence respectively. Let us recall that $x_n\weakly x$ if $\scal{x_n-x}{y} \to 0$ for all $y\in\HH$. 

\textbf{Subdifferential.} 
$\partial f$ denotes the subdifferential of $f$,
\begin{equation*}
 \partial f(\xx) \!=\! \menge{\uu\!\in\!\RR^d}{\!f(\yy)\! -\!f(\xx)\!\geq\!\scal{\yy\!-\!\xx\!}{\!\uu}, \forall \yy\!\in\!\RR^d}.
 \end{equation*}

\textbf{Proximal operator.} 
The proximal operator of $f$ is defined as 
\begin{equation*}
\prox_f (\xx) = \arg\min_{\zz \in \RR^d} \left\{ f(\zz) + \frac{1}{2} \| \zz - \xx \|^2 \right\}.
\end{equation*}

\textbf{Domain, graph, zeros and range.} 
Domain, graph, range and the set of zeros of $\AAA$ are defined as follows:
\begin{align*}
\dom \AAA & =\menge{\xx \in \HHH}{\AAA \xx \neq\emp} \\
\gra \AAA & =\menge{(\xx,\uu)\in\HHH\times\HHH}{u\in \AAA \xx} \\
\ran \AAA & =\menge{\uu\in\HHH}{(\exists \xx\in\HHH)\; u\in \AAA \xx} \\
\zer \AAA &=\menge{\xx\in\HHH}{0\in \AAA \xx}
\end{align*}

\textbf{Inverse.} We denote the inverse of $\AAA$ by $\AAA^{-1}$:
\begin{equation*}
 \xx \in \AAA^{-1}\uu \iff \uu \in \AAA \xx
\end{equation*}

\textbf{Resolvent.} 
The resolvent of $\AAA$ is defined as
\begin{equation}
\label{eq:res}
J_{\AAA}=(\Id+\AAA)^{-1}
\end{equation}
where $\Id$ is the identity operator of $\HHH$. 
When $\AAA = \partial f$, 
$J_{\AAA} = \prox_{f}$.

\textbf{Monotone operator.} 
$\AAA$ is said to be a monotone operator if
\begin{equation*}
\scal{\xx-\yy}{\uu-\vv} \geq 0 \quad\quad
\begin{aligned}
\forall (\xx,\uu) &\in\gra \AAA, \\
\forall (\yy,\vv) &\in\gra \AAA.
\end{aligned}
\end{equation*}

\textbf{Maximally monotone operator.} 
$\AAA$ is maximally monotone if $\AAA$ is monotone and if there exists no monotone operator $\widetilde{\AAA}\colon\HHH\to\HHH$
such that $\gra \AAA\subset \gra \widetilde{\AAA}\neq \gra \AAA$. 

\textbf{Uniformly monotone operator.} 
$\AAA$ is uniformly monotone at $\yy$ if there exists a function $\phi\colon \left[0,+\infty\right[\to \left[0,+\infty\right]$
vanishing only at $0$ such that 
\begin{equation*}
\scal{\xx-\yy}{\uu-\vv} \geq \phi(\|\yy-\xx\|) \quad\quad
\begin{aligned}
\forall \vv &\in \AAA\yy, \\
\forall (\xx,\uu) &\in\gra \AAA.
\end{aligned}
\end{equation*}

\textbf{Fixed points.} 
We denote the set fixed points of an operator $T\colon\HHH\to\HHH$ as
\begin{equation*}
\Fix(T) = \menge{\xx\in\HHH}{\xx=T\xx}.
\end{equation*}

\textbf{Non-expansive operator.} 
An operator $T\colon\HHH\to\HHH$ is non-expansive if
\begin{equation*}
\|T\xx-T\yy\| \leq \|\xx-\yy\| \quad\quad
\forall \xx \in\HHH, ~
\forall \yy \in\HHH.
\end{equation*}

\textbf{Averaged operator.} Let $\alpha \in ]0,1[$. 
An operator $T\colon\HHH\to\HHH$ is $\alpha$-averaged if
$T = (1-\alpha)\Id + \alpha R$
for some non-expansive operator $R\colon\HHH\to\HHH$.

\textbf{Demiregular operator} {\cite[Definition 2.3]{Attouch2010}}\textbf{.} An operator $\AAA$ is demiregular at $\yy\in\dom(\AAA)$ if,
for every sequence $(\yy_n,\vv_n)_{n\in\NN} $ in $\gra \AAA$ and every $\vv\in \AAA\yy$ such that $\yy_n\weakly \yy$ and $\vv_n\to \vv$, we have $\yy_n\to \yy$.

 Next, we present 3 key lemmas to be used in the proof of the main convergence theorem.

\begin{lemma}{\em(See \cite[Lemma 2.2]{Davis2017threeop})}\label{l:0}
 Let $\gamma$ be a strictly positive number.  
 Define $T$ as follows:
 \begin{equation}
 \label{eqn:operator-T}
 T = \Id-J_{\gamma \BB} + J_{\gamma \AAA}\circ(2J_{\gamma \BB}-
Id -\gamma \QQ\circ J_{\gamma \BB}).
 \end{equation}
 Then, $\Fix(T)\not=\emp$ whenever $\zer(\AAA+\BB +\QQ)\not=\emp$. 
 Furthermore, $\zer(\AAA+\BB+\QQ) = J_{\gamma \BB}(\Fix(T))$.
\end{lemma}

\begin{lemma}{\em(See \cite[Lemma 2.3]{Alvarez2004ihp})} \label{l:alva}
Let $(s_n)_{n\in\NN}$ and $(\delta_n)_{n\in\NN}$ be a nonnegative sequence such that 
$\sum_{n\in\NN}\delta_n <+\infty$ and $s_{n+1} \leq s_n +\alpha_n(s_n-s_{n-1}) +\delta_n$, where 
$(\alpha_n)_{n\in\NN}\in \left[0,\alpha\right]^{\NN}$, for some $\alpha\in \left]0,1\right[$. Then the followings hold:
\begin{enumerate}
\item $\sum_{n=1}^{\infty} \max\{s_n-s_{n-1},0\} <+\infty$.
\item There exists $s^*\in \left[0,+\infty \right[$ such that $s_n\to s^*$.
\end{enumerate}
\end{lemma}

\begin{lemma}\label{lem:01}
Let $\mathcal{M}$ be a non-empty closed affine subset of $\HH$, 
and $T\colon\mathcal{M}\to \mathcal{M}$ be an $\alpha$-averaged operator for some $\alpha\in \left]0,1\right[$ such that $\Fix(T)\not=\emp$. 
Consider the following iterative scheme:
\begin{equation}
\label{e:proof21}
(\forall n\in\NN_{+})\quad
\begin{array}{l}
\left\lfloor
\begin{array}{l}
w_n = \overline{x}_n +\tau_n(\overline{x}_n-\overline{x}_{n-1})\\
\overline{x}_{n+1} = w_n+ \lambda_n\big( Tw_n -w_n\big).
\end{array}
\right.\\[2mm]
\end{array}
\end{equation}
Let $\overline{x}_0, \overline{x}_1 \in \mathcal{M}$, let $(\tau,\varepsilon) \in \left[0,1\right[^2$, and let $(\tau_n)_{n\in \NN_+}$ be a nondecreasing sequence in $\left[0,\tau\right]^{\NN_+}$ with $\tau_1= 0$. 
Let $(\lambda_n)_{n\in\NN_+}$ be a strictly positive sequence such that $\lambda_n\geq \varepsilon$ for all $n$.
Let $\delta > 0$ and $\sigma > 0$ be such that 
\begin{equation*}
\delta > \tfrac{\tau^2(1+\tau) +\tau\sigma}{1-\tau^2}\quad \text{and}\quad 
\varepsilon \leq \lambda_n\leq \tfrac{\delta - \tau(\tau +\tau^2 +\tau\delta +\sigma)}{\alpha\delta(1+\tau +\tau^2 +\tau\delta +\sigma )}.
\end{equation*}
Then the followings hold:
\begin{enumerate}
\item $\sum_{n\in\NN}\|\overline{x}_{n+1}- \overline{x}_n\|^2 <+\infty$.
\item $(\overline{x}_n)_{n\in\NN}$ converges weakly to a point in $\Fix(T)$. 
\end{enumerate}
\end{lemma}

\begin{proof} 
This lemma is a direct consequence of {\rm \cite[Theorem 5]{Bot2015iner}}.
Define $R = (1-\alpha^{-1})\Id +\alpha^{-1} T$, and set $\mu_n =\alpha\lambda_n$. 
Then, we can rewrite \eqref{e:proof21} as:
\begin{equation*}
(\forall n\in\NN_+)\quad
\begin{array}{l}
\left\lfloor
\begin{array}{l}
w_n = \overline{x}_n +\tau_n(\overline{x}_n-\overline{x}_{n-1})\\
\overline{x}_{n+1} = w_n+ \mu_n\big( Rw_n -w_n\big).
\end{array}
\right.\\[2mm]
\end{array}
\end{equation*}
It is easy to verify that $R$ and $(\mu)_{n\in\NN_+}$ satisfy all conditions in {\rm \cite[Theorem~5]{Bot2015iner}}. The proof directly follows from there. 
\end{proof}

\begin{remark} Suppose that $1+\tau_n$ and $-\tau_n$ are non-negative such that $\inf_{n\in\NN}(1+\tau_n) > 0$. Set $\lambda_n\equiv 1/\alpha$.
Then it is shown in \cite[Example~4.3]{Combettes2017} that $\|w_n-Rw_n\| = \alpha^{-1} \|w_n-Tw_n\| \to 0$. Moreover, there exists $w\in \Fix(T)$ such that 
$w_n\weakly w$ and $\overline{x}_n\weakly w$.
\end{remark}

\section{Algorithm \& Convergence}
\label{s:algo}

We describe the inertial forward-Douglas-Rachford splitting method (IFDR) for solving Problem \ref{prob:monotone-inclusion} in \Cref{alg:I3OSM}, 
and we prove the weak convergence of the proposed method in \Cref{thm:IFDR-monotone-inclusion}.

\begin{algorithm}[H]
   \caption{IFDR for Problem~\ref{prob:monotone-inclusion}}
   \label{alg:I3OSM}
\begin{algorithmic}
\vspace{1mm}
   \STATE {\bfseries Input:} initial points $\overline{\xx}_0=\overline{\xx}_1$ in $\HHH$, step size $\gamma$, two sequences of strictly positive numbers $(\tau_n)_{n\in\NN_{+}}$ and $(\lambda_n)_{n\in\NN_{+}}$.
   \STATE {\bfseries Procedure:} 
   \FOR{$n=1,2,\ldots$}
   \STATE $\ww_n = \overline{\xx}_n +\tau_n(\overline{\xx}_n-\overline{\xx}_{n-1})$
   \STATE $\xx_n = J_{\gamma\BB}\ww_n$
   \STATE $\yy_n= J_{\gamma\AAA} (2\xx_n-\ww_n - \gamma\QQ\xx_n)$
   \STATE $\overline{\xx}_{n+1} = \ww_n +\lambda_n\big(\yy_n -\xx_n\big)$
   \ENDFOR 
\vspace{1mm}
\end{algorithmic}
\end{algorithm}

\begin{theorem} 
\label{thm:IFDR-monotone-inclusion}
Suppose that the parameter $\gamma$ and the sequences $(\lambda_n)_{n\in\NN_+}$ and $(\tau_n)_{n\in\NN_+}$ satisfy the following conditions:
\begin{enumerate}[(i).]
\item $\gamma \in \left]0,2\beta\kappa\right[$
\item $\tau \geq \tau_{n+1} \geq \tau_{n} \quad$ for all $n\in\NN_+$ 
\item $\epsilon < \lambda_n\leq \frac{\delta - \tau(\tau +\tau^2 +\tau\delta +\sigma)}{\alpha\delta(1+\tau +\tau^2 +\tau\delta +\sigma )}\quad$ for all $n\in\NN_+ \quad$ where $\quad \alpha = \tfrac{2\beta}{4\beta-\gamma}$  
\end{enumerate}
for some $\quad (\tau,\epsilon,\kappa) \in \left]0, 1\right[^3\quad$ and $ \quad (\delta, \sigma) \in \left]0,+\infty\right[^2 \quad$ that satisfy $\quad   \delta > \tfrac{\tau^2(1+\tau) +\tau\sigma}{1-\tau^2}$. 

\noindent Then, there exists a point $\ww\in\HHH$ such that the followings hold:
\begin{enumerate}
\item\label{conc:i} $\sum_{n\in\NN_+}\|\overline{\xx}_{n+1}- \overline{\xx}_n\|^2 <+\infty$.
\item \label{conc:ii}
$(\overline{\xx}_n)_{n\in\NN_+}$ converges weakly to $\ww$.
\item\label{conc:iii}  $(\xx_n)_{n\in\NN_+}$ converges weakly to $\xx^\star = J_{\gamma\BB}\ww \in \boldsymbol{\mathcal{X}^\star}$.
\item \label{conc:iv} Suppose that one of the followings holds:
\begin{enumerate}[(a).]
\item \label{conc:iva}$\QQ$ is demiregular at $\xx^\star$.
\item \label{conc:ivb} $\AAA$ is uniformly monotone at $\xx^\star$.
\item\label{conc:ivc} $\BB$ is uniformly monotone at $\xx^\star$.
\end{enumerate}
Then, $\xx_n$ converges to $\xx^\star$ almost surely.
\end{enumerate}
\end{theorem}

\begin{remark}
Similar conditions relating the step-sizes $\gamma$ and $\lambda_n$ to the inertia parameter $\tau_n$ are considered in the inertial Douglas-Rachford splitting \cite{Bot2015iner}.
\end{remark}

\begin{remark} 
When $g = 0$
, IFDR reduces to the standard inertial proximal point method (\textit{cf.}, \cite{attouch2016,attouch2017,Beck2009fista,Chambolle2015,Bonettini2017,Aujol2015}). 
 And for the choice of $\tau_n \equiv 0$, IFDR reduces to the three operator splitting method in \cite{Davis2017threeop}. 
\end{remark}

\begin{remark} \label{rem:negativetau}
If $\QQ = \Id$, 
we can chose $\tau_n$ such that $1+\tau_n$ and $\-\tau_n$ such that $\inf_{n\in\NN}(1+\tau_n) > 0$ and $\lambda_n= 1$. Then above results remains valid for any positive $\gamma$.
\end{remark}

\subsection{Proof of \Cref{thm:IFDR-monotone-inclusion}}
Let $T$ be defined as \eqref{eqn:operator-T}, then the iterative updates of IFDR can be written as \eqref{e:proof21}.
It follows from \cite[Proposition 3.1]{Davis2017threeop} that 
$T$ is an $\alpha$-averaged operator with $\alpha= 2\beta(4\beta-\gamma)^{-1}$.
The conditions of \Cref{lem:01} also satisfy  the conditions listed in \Cref{thm:IFDR-monotone-inclusion}. 

\ref{conc:i}$\&$ \ref{conc:ii}: These  follow from Lemma \ref{lem:01} with $\ww \in\Fix(T)$.

\ref{conc:iii}: Since $\ww\in\Fix(T)$, we have $\xx= J_{\gamma \BB}\ww \in \boldsymbol{\mathcal{X}^\star}$
and $\ww = T\ww$. 
It follows from \cite[Corollary 2.14]{Bauschke2011} that 
{
\medmuskip=1mu
\thinmuskip=1mu
\thickmuskip=1mu
\nulldelimiterspace=2pt
\scriptspace=1pt
\begin{alignat}{2}
\|\overline{\xx}_{n+1}-\ww\|^2 &=   \|(1\!-\!\lambda_n)(\ww_n-\ww) +\lambda_n( T\ww_n-T\ww) \|^2\notag\\
& = (1-\lambda_n) \|\ww_n-\ww \|^2 +\lambda_n \|T\ww_n-T\ww \|^2 \notag \\
& \quad\quad\quad\quad\quad\quad - \lambda_n(1-\lambda_n) \|T\ww_n- \ww_n \|^2. \notag
\end{alignat}
}
As it is shown in \cite[Eq. (2.3)]{Davis2017threeop} that 
\begin{equation*}
\begin{aligned}
\|T\ww_n-T\ww \|^2  \leq & \|\ww_n-\ww \|^2  - \tfrac{1-\alpha}{\alpha} \|T\ww_n-\ww_n \|^2 \\
& \quad\quad -\gamma(2\beta - \tfrac{\gamma}{\kappa})\|\QQ\xx_n- \QQ\xx \|^2.
\end{aligned}
\end{equation*}
Therefore, 
\begin{equation}
\label{e:bais}
\begin{aligned}
\|\overline{\xx}_{n+1}-\ww\|^2\leq & \|\ww_n-\ww \|^2  - \rho_{1,n}\|R\ww_n-\ww_n \|^2 \\ 
& \quad\quad  -\rho_{2,n}  \|\QQ\xx_n- \QQ\xx \|^2,
\end{aligned}
\end{equation}
where we set
\begin{equation*}
\rho_{1,n} = \lambda_n(1-\lambda_n)+\lambda_n\tfrac{1-\alpha}{\alpha} ~~
\text{and}~~ \rho_{2,n} = \lambda_n\gamma(2\beta-\tfrac{\gamma}{\kappa}).
\end{equation*}
Let us estimate two first terms in the right hand side of \eqref{e:bais}. Using \cite[Corollary 2.14]{Bauschke2011}, we have
\begin{equation} 
\label{e:es1}
\begin{aligned}
\|\ww_n-\overline{\xx} \|^2 & = (1+\tau_n) \|\overline{\xx}_n-\ww \|^2 - \tau_n \|\overline{\xx}_{n-1}-\ww \|^2  \\ 
& \quad\quad +\tau_n(1+\tau_n) \|\overline{\xx}_n-\overline{\xx}_{n-1} \|^2,
\end{aligned}
\end{equation}
and  upon setting $\mu_n = \alpha\lambda_n$ and $\rho_n = (\tau_n+\delta\mu_n)^{-1}$, 
\begin{alignat}{2}
\label{e:es2}
& \|T\ww_n-\ww_n \|^2 = \tfrac{1}{\lambda_{n}^2} \|\overline{\xx}_{n+1}-\ww_n \|^2\notag\\
& \quad\quad =  \tfrac{1}{\lambda_{n}^2} \big(\|\overline{\xx}_{n+1}-\overline{\xx}_n +\tau_n (\overline{\xx}_{n-1}-\overline{\xx}_n) \|^2 \big)\notag\\
&\quad\quad= \tfrac{1}{\lambda_{n}^2} \|\overline{\xx}_{n+1}-\overline{\xx}_n\|^2 + \tfrac{\tau^{2}_n}{\lambda_{n}^2} \|\overline{\xx}_{n-1}-\overline{\xx}_n\|^2 \notag \\ 
& \quad\quad\quad +2 \tfrac{\tau_n}{\lambda_{n}^2} \scal{\overline{\xx}_{n+1}-\overline{\xx}_n}{\overline{\xx}_{n-1}-\overline{\xx}_n} \\
&\quad\quad \geq \tfrac{1}{\lambda_{n}^2} \|\overline{\xx}_{n+1}-\overline{\xx}_n\|^2 +\tfrac{\tau^{2}_n}{\lambda_{n}^2} \|\overline{\xx}_{n-1}-\overline{\xx}_n\|^2 
\notag \\
&\quad\quad\quad - \tfrac{\tau_n\rho_n}{\lambda_{n}^2}\|\overline{\xx}_{n+1}-\overline{\xx}_n\|^2 -\tfrac{\tau_n}{\lambda_{n}^2\rho_n}\| \overline{\xx}_{n-1}-\overline{\xx}_n \|^2. \notag
\end{alignat}
Set 
\begin{equation*}
\begin{cases}
 \chi_{1,n} =\rho_{1,n}\lambda_{n}^{-2}(\alpha_n\rho_n- 1)\\
\chi_{2,n} = \tau_n(1+\tau_n)+\rho_{1,n}\lambda_{n}^{-2}\rho_{n}^{-1}\tau_n(1-\rho_n\tau_n)
\end{cases}
\end{equation*}
Then, inserting \eqref{e:es1} and \eqref{e:es2} into \eqref{e:bais}, we get
{
{
\medmuskip=1mu
\thinmuskip=1mu
\thickmuskip=1mu
\nulldelimiterspace=2pt
\scriptspace=1pt
\begin{alignat}{2}
\|\overline{\xx}_{n+1}-\ww\|^2& \leq  (1+\tau_n) \|\overline{\xx}_n-\ww\|^2 - \tau_n \|\overline{\xx}_{n-1}-\ww \|^2 \notag\\ 
&~~ - \rho_{2,n}\|\QQ\xx_n\!-\!\QQ\xx \|^2 +\chi_{1,n} \|\overline{\xx}_{n+1}\!-\!\overline{\xx}_n\|^2 \notag\\ 
&~~  +\chi_{2,n} \| \overline{\xx}_{n-1}\!-\!\overline{\xx}_n \|^2. \notag
\end{alignat}
}
Simple calculations show that 
\begin{equation*}
\begin{cases}
 \chi_{1,n} =  \frac{1-\mu_n}{\mu_n}(\alpha_n\rho_n-1) \\
\chi_{2,n} = \tau_n(1+\tau_n)+\frac{1-\mu_n}{\mu_n\rho_n}(1-\alpha_n\rho_n),
\end{cases}
\end{equation*}
and hence under the conditions on $\gamma$ and $\lambda_n$, the two sequences $(\chi_{1,n})_{n\in\NN_+}$ and $(\chi_{2,n})_{n\in\NN_+}$ are uniformly bounded. 
In view of the result \ref{conc:i},
$(\chi_{1,n} \|\overline{\xx}_{n+1}-\overline{\xx}_n\|^2  +\chi_{2,n} \| \overline{\xx}_{n-1}-\overline{\xx}_n \|^2)_{n\in\NN_+}$ is summable. By Lemma \ref{l:alva},  we have 
\begin{equation}
\label{e:cs1}
\begin{cases}
\text{ $(\|\overline{\xx}_n-\ww \|)_{n\in\NN}$ converges}\\
\sum_{n\geq 1} \max\{\|\overline{\xx}_n-\ww\| -\|\overline{\xx}_{n-1}-\ww\|,  0\} <+\infty,
\end{cases}
\end{equation}
hence, it follows that  
\begin{equation}
\sum_{n\in\NN_+}  \rho_{2,n}\|\QQ\xx_n-\QQ\xx \|^2  <+\infty.
\end{equation}
Since $(\rho_{2,n})_{n\in\NN_+}$ is bounded away from zero, we have $\QQ\xx_n\to \QQ\xx$. Moreover, it follows from \eqref{e:cs1} that 
$(\overline{\xx}_n)_{n\in\NN_+}$ is bounded, that, together with the boundedness of $(\tau_n)_{n\in\NN_+}$, implies that $(\ww_n)_{n\in\NN_+}$ is bounded. 

Since $J_{\gamma \BB}$ 
is non-expansive, it follows that $(\xx_n)_{n\in\NN_+}$ is bounded. Now, let $\xx^\star$ be a weak cluster point of $(\xx_n)_{n\in\NN_+}$, i.e., there exists a subsequence $(\xx_{k_n})_{n\in\NN_+}$ of $(\xx_n)_{n\in\NN_+}$ such that $\xx_{k_n}\weakly \xx^\star$. Since $\QQ$ is maximally monotone and $\QQ\xx_{k_n}\to\QQ\xx^\star$, it follows from 
\cite[Proposition 20.33(ii)]{Bauschke2011} that $\QQ\xx^\star = \QQ\overline{\xx}$ and hence $\QQ\xx_{k_n}\to \QQ\xx$.
Note that 
\begin{equation*}
\|T\ww_n-\ww_n\|^2 \leq \tfrac{2}{\lambda_{n}^2}(\|\overline{\xx}_{n+1}-\overline{\xx}_n \|^2 + \|\overline{\xx}_{n}-\overline{\xx}_{n-1} \|^2) \to 0.
\end{equation*}
Therefore, by setting $\yy_n = J_{\gamma \AAA}(2\xx_n-\ww_n-\gamma \QQ\xx_n)$, we have $\xx_n-\yy_n\to 0$ and hence $\yy_n\weakly \xx^\star$. To 
sum up, we have
{
\medmuskip=1mu
\thinmuskip=1mu
\thickmuskip=1mu
\nulldelimiterspace=2pt
\scriptspace=1pt
\begin{equation}
\label{e:gr1}
\begin{cases}
 \ww_{k_n}-\xx_{k_n} \in \gamma \BB\xx_{k_n}; ~  \ww_{k_n}-\xx_{k_n} \weakly \ww -\xx^\star; ~ \xx_{k_n}\weakly \xx^\star\\
 2\xx_{k_n}-\ww_{k_n}-\gamma \QQ\xx_{k_n} -\yy_{k_n} \in \gamma \AAA \yy_{k_n}\\
  2\xx_{k_n}-\ww_{k_n}-\gamma \QQ\xx_{k_n} -\yy_{k_n}\weakly \xx^\star-\ww -\gamma \QQ\xx ; ~ \yy_{k_n}\weakly \xx^\star\\
  \gamma\QQ\xx_{k_n} \in \gamma\QQ\xx_{k_n} \to \gamma\QQ \xx^\star\\
  \xx_{k_n}-\yy_{k_n}\to 0.
 \end{cases}
\end{equation}
}
Therefore, by \cite[Proposition 25.5]{Bauschke2011}, we have 
\begin{equation}
\label{e:gr2} 
\ww -\xx^\star\in\gamma \BB\xx^\star\; \quad\text{and}\quad\; \xx^\star-\ww -\gamma \QQ\xx\in \gamma\AAA\xx^\star,
\end{equation}
which implies that $\xx^\star = J_{\gamma \BB}\ww$ and it is the unique cluster point of $(\xx_n)_{n\in\NN_+}$. Now, by \cite[Lemma 2.38]{Bauschke2011}, we obtain
$\xx_n\weakly J_{\gamma \BB}\ww$.
  
  \ref{conc:iva}:  Since $\xx_n\!\weakly\!\xx^\star$ and $\QQ\xx_n\!\to\!\QQ\xx^\star$, and $\QQ$ is demiregular at $\xx$, by definition, it follows that $\xx_n\to\xx^\star$.  
  
    \ref{conc:ivb}: In view of \eqref{e:gr1} and \eqref{e:gr2}, we have 
    \begin{equation*}
    \begin{array}{c}
   2\xx_{n}-\ww_{n}-\gamma \QQ\xx_{n} -\yy_{n} \in \gamma \AAA \yy_{n} \\ 
   \text{and}\quad \xx^\star-\ww -\gamma \QQ\xx^\star\in \gamma\AAA\xx^\star.
	\end{array}
    \end{equation*}
    Sicne $\AAA$ is uniformly monotone at $\xx$, there exists an increasing function $\phi\colon \left[0,+\infty\right[\to \left[0,+\infty\right]$ vanishing only at $0$ 
    such that 
    \begin{alignat}{2}
    \label{e:pr3}
   &\gamma \phi(\|\yy_n-\xx^\star \|)  \notag \\
    &~~ \leq \scal{ 2\xx_{n}-\ww_{n}- \QQ\xx_{n} -\yy_{n} -\xx^\star+\ww +\QQ\xx }{\yy_n-\xx^\star}\notag\\
    &~~= \scal{\yy_n-\xx^\star }{\QQ\xx_{n} -\QQ\xx^\star} \notag \\
    &~~\quad\quad+\scal{\yy_n-\xx^\star}{ 2\xx_{n}-\ww_{n}-\yy_{n} -\xx^\star+\ww}\notag\\
    &~~= t_{1,n} +t_{2,n},
    \end{alignat}
    where we set 
    \begin{equation*}
    \begin{cases} 
    t_{1,n} = \scal{\yy_n-\xx^\star }{\QQ\xx_{n} -\QQ\xx^\star}\notag\\
    t_{2,n} = \scal{\yy_n-\xx^\star}{ 2\xx_{n}-\ww_{n}-\yy_{n} -\xx^\star+\ww}\notag\\
    t_{3,n}= \scal{\yy_n-\xx^\star}{\xx_n -\yy_n}\\
    t_{4,n} = \scal{\yy_n-\xx_n}{\xx_{n}-\ww_{n}-\yy_{n} -\xx^\star+\ww}
    \end{cases}
    \end{equation*}
    We next estimate $(t_{1,n})_{n\in\NN_+}$ and $(t_{2,n})_{n\in\NN_+}$ in the right hand side of \eqref{e:pr3}. Since $\yy_n-\xx^\star\weakly 0$, it is bounded, 
    an since $\QQ\xx_n\to\QQ\xx^\star$, we have 
    \begin{equation*}
    |t_{1,n}| \leq \|\yy_n-\xx^\star\| \|\QQ\xx_n-\QQ\xx^\star\| \to 0.
    \end{equation*}  
    Using the monotonicity of $\BB$, we also have 
    \begin{equation*}
    \scal{\xx_n-\xx^\star}{\xx_{n}-\ww_{n} -\xx^\star+\ww} \leq 0,
    \end{equation*}
    and hence
    \begin{alignat}{2}
    t_{2,n}& = t_{3,n} +\scal{\yy_n-\xx^\star}{\xx_{n}-\ww_{n}-\yy_{n} -\xx^\star+\ww}\notag\\
    &= t_{3,n} +t_{4,n} +\scal{\xx_n-\xx^\star}{\xx_{n}-\ww_{n} -\xx^\star+\ww}\notag\\
    &\leq t_{3,n} +t_{4,n} \notag\\
    &\leq |t_{3,n}| +|t_{4,n}| \quad\to\quad 0. \notag
    \end{alignat}
    Therefore, we derive from \eqref{e:pr3} that 
    \begin{alignat}{2}
     \gamma \phi(\|\yy_n-\xx^\star \|)& \leq |t_{1,n}| +t_{2,n} \notag \\
     & \leq |t_{1,n}| +|t_{3,n}| +|t_{4,n}| \quad\to\quad 0, \notag
    \end{alignat}
    which implies that $\yy_n\to\xx^\star$ and hence $\xx_n\to \xx^\star$.
    
     \ref{conc:ivc}: Suppose that $\BB$ is uniformly monotone at $\xx^\star$, then 
      there exists an increasing function $\psi\colon \left[0,+\infty\right[\to \left[0,+\infty\right]$ vanishing only at $0$ 
    such that 
\begin{alignat}{2}
\gamma \psi(\|\xx_n-\xx^\star\|) &\leq \scal{\xx_n-\xx^\star}{\ww_{n}-\xx_n +\xx^\star-\ww}\notag\\
 &= t_{3,n} +t_{4,n} + t_{2,n}\notag\\
 &\leq 2|t_{3,n}| +|t_{4,n}| \quad\to\quad 0 \notag
\end{alignat}
which implies that $\xx_n\to \xx^\star$. \endproof 

\section{Convex Optimization Applications}
\label{s:convex-optimization}

In this section, we present the special instance of \Cref{alg:I3OSM} that applies to Problem~\ref{prob:convex-optimization}. 

\begin{remark}
\label{rem:mono-to-convex}
Problem \ref{prob:convex-optimization} is a special case of Problem \ref{prob:monotone-inclusion}, with $\AAA = \partial{f},~\BB = \partial{g}$, $\QQ = \nabla h$ and $\beta = L^{-1}$ in $\HHH = \mathbb{R}^d$.
\end{remark}

\begin{algorithm}[h]
   \caption{IFDR for Problem~\ref{prob:convex-optimization}}
   \label{alg:I3OSM-convex}
\begin{algorithmic}
\vspace{1mm}
   \STATE {\bfseries Input:} initial points $\overline{\xx}_0 = \overline{\xx}_1$ in $\RR^d$, step size $\gamma$, two sequences of strictly positive numbers $(\tau_n)_{n\in\NN_{+}}$ and $(\lambda_n)_{n\in\NN_{+}}$.
   \STATE {\bfseries Procedure:} 
   \FOR{$n=1,2,\ldots$}
   \STATE $\ww_n = \overline{\xx}_n +\tau_n(\overline{\xx}_n-\overline{\xx}_{n-1})$
   \STATE $\xx_n = \prox_{\gamma g}\ww_n$
   \STATE $\yy_n=\prox_{\gamma f}(2\xx_n-\ww_n - \gamma\nabla h(\xx_n))$
   \STATE $\overline{\xx}_{n+1} = \ww_n +\lambda_n\big(\yy_n -\xx_n\big)$
   \ENDFOR
   \vspace{1mm}
\end{algorithmic}
\end{algorithm}
   \vspace{-.5mm}

\begin{theorem} 
\label{thm:IFDR-convex-optimization}
Suppose that the parameter $\gamma$ and the sequences $(\lambda_n)_{n\in\NN_+}$ and $(\tau_n)_{n\in\NN_+}$ satisfy the conditions given in Theorem \ref{thm:IFDR-monotone-inclusion} with $\beta = L^{-1}$. 
Then, there exists a point $\ww\in\RR^d$ such that the followings hold:
\begin{enumerate}
\item\label{conc:i} $\sum_{n\in\NN_+}\|\overline{\xx}_{n+1}- \overline{\xx}_n\|^2 <+\infty$.
\item \label{conc:ii}
$(\overline{\xx}_n)_{n\in\NN_+}$ converges to  $\ww$.
\item\label{conc:iii}  $(\xx_n)_{n\in\NN_+}$ converges to a solution $\xx^\star = \prox_{\gamma g}\ww$. 
\end{enumerate}
\end{theorem}

\begin{proof}
Follows \Cref{thm:IFDR-monotone-inclusion} in view of \Cref{rem:mono-to-convex}.
\end{proof}

\subsection{IFDR for multivariate minimization}

	Let $m$ be a strictly positive integer, and $L$ be a strictly positive real number. 
	For every $i\in\{1,\ldots,m\}$, let $d_i$ be a strictly positive integer and $f_i\colon \RR^{d_i}\to \left]-\infty,+\infty \right]$ be proper lower semicontinuous functions.  
	Suppose that $\varphi \colon \RR^{d_1}\times\ldots\times \RR^{d_m}\to\RR $ is a differentiable convex function with $L$-Lipschitz continuous gradient. 	
	We consider the following multivariate minimization problem:
	\begin{equation*}
\begin{aligned} 
& \underset{(x_i \in \mathbb{R}^{d_i})_{1\leq i\leq m}}{\text{minimize}} & \varphi(\xx) + \sum_{i=1}^m  (f(x_i) + g(x_i) )
\end{aligned}
\end{equation*}

	We denote by $\nabla_i\varphi$ the i$^{th}$ component of $\nabla\varphi$.
	Suppose that the set $\boldsymbol{\mathcal{X}}$ of all point 
	${\xx}= ({x}_1,\ldots,{x}_m)$ to the following coupled system of inclusion is non-empty:
	\begin{equation}
	\label{e:sys}
	\begin{cases}
	0&\in \partial f_1(x_1)+ \partial g_1(x_1) + \nabla_1\varphi({\xx})\notag\\
	\vdots\\
		0&\in \partial f_m({x}_m)+ \partial g_m({x}_m) +\nabla_m \varphi ({\xx})\\
		\end{cases}
	\end{equation}
		
\begin{algorithm}[t]
   \caption{IFDR for multivariate minimization}
   \label{alg:IFDRmulti}
\begin{algorithmic}
\vspace{1mm}
   \STATE {\bfseries Input:} initial points $\overline{\xx}_0 = \overline{\xx}_1$ in $\RR^d$, step size $\gamma$, two sequences of strictly positive numbers $(\tau_n)_{n\in\NN_{+}}$ and $(\lambda_n)_{n\in\NN_{+}}$.
   \STATE {\bfseries Procedure:} 
   \FOR{$n=1,2,\ldots$}
   \FOR{$i=1,2,\ldots,m$}
   \STATE $w_{i,n} = \overline{x}_{i,n} +\tau_n(\overline{x}_{i,n}-\overline{x}_{i,n-1})$
   \STATE $x_{i,n} = \prox_{\gamma g_i}w_{i,n}$
   \ENDFOR
   \FOR{$i=1,2,\ldots,m$}
   \STATE $y_{i,n}=\prox_{\gamma f}(2x_{i,n}-w_{i,n} - \gamma\nabla \varphi(\xx_n))$
   \ENDFOR
   \STATE $\overline{\xx}_{n+1} = \ww_n +\lambda_n\big(\yy_n -\xx_n\big)$
   \ENDFOR
   \vspace{1mm}
\end{algorithmic}
\end{algorithm}

Suppose that the parameters $\gamma$, $(\lambda_n)_{n\in\NN_+}$ and $(\tau_n)_{n\in\NN_+}$ satisfy the conditions of \Cref{thm:IFDR-monotone-inclusion} with $\beta = 1/L$.
Then, for each $i\in\{1,\ldots,m\}$, there exists $\overline{x}_i\in\RR^{d_i}$ such that the following hold.
\begin{enumerate}
\item\label{c:0i} $\sum_{n\in\NN}\|\overline{x}_{i,n+1}- \overline{x}_{i,n}\|^2 <+\infty$.
\item \label{c:0ii}
$(\overline{x}_{i,n})_{n\in\NN}$ converges to a point $\overline{x}_i$.
\item\label{c:0iii}  $(x_{i,n})_{n\in\NN_+}$ converges  to a point $x_i^\star = J_{\gamma B_i}\overline{x}_i $ and $\xx^\star= (x_1^\star,\ldots,x_m^\star)\in\boldsymbol{\mathcal{X}^\star}$.
\end{enumerate}

   \begin{remark} 
   When $g=0$ and $\tau_n \equiv 0$, \Cref{alg:IFDRmulti} reduces to the one proposed in \cite{Combettes2012}.
 \end{remark}

 \section{Adaptive Restart}
 \label{sec:restart}

The choice of the inertia parameter $(\tau_n)_{n \in \NN_+}$ directly affects the performance of IFDR. 
In practice, we observe that the parameter $\tau$ imposed by \Cref{thm:IFDR-monotone-inclusion} is too conservative, 
in the sense that some choices $\tau_n > \tau$ perform better in practice. 

In this section, we propose a heuristic adaptive restart technique for choosing a practical inertia parameter. 
The proposed scheme outperforms other methods in our experiments (cf. \Cref{s:numer}). 

\begin{algorithm}[h]
   \caption{IFDR with restart (IFDR-R)}
   \label{alg:I3OSM-R}
\begin{algorithmic}
\vspace{1mm}
   \STATE {\bfseries Input:} {Input of IFDR, a function\footnotemark ~$\psi(\cdot):\RR^d\to\RR$}.
   \STATE {\bfseries Procedure:} 
   \STATE $ t = 1$
   \FOR{$n=1,2,\ldots$}
   \STATE {Apply one iteration of IFDR with $\tau_n = \tfrac{n-t}{n+3-t}$}
   \IF {$\psi(\xx_n) \geq \psi(\xx_{n-1})$} 
   	\STATE {Set  $t = n$ and $\tau_n = 0$}
	 \STATE {Recompute last iteration with new $\tau_n$}
   \ENDIF
   \ENDFOR
   \vspace{1mm}
\end{algorithmic}
\end{algorithm}

If both of the nonsmooth terms $f$ and $g$ in Problem~\ref{prob:convex-optimization} are Lipschitz continuous, a natural choice in \Cref{alg:I3OSM-R} would be $\psi = (f+g+h)$. 
If one of them is a constraint indicator function, we can ensure the feasibility of $\xx_n$ by choosing this term as $g$ in our template. 
In this case, $g=0$ for all $\xx_n$ hence the natural choice is $\psi = (f+h)$. 
If both of the terms are indicator functions, we recommend the following convention:
\begin{equation*}
\rotatebox[origin=c]{90}{\text{restart iff~}}~
\begin{cases}
 ~ f(\xx_n) = \infty, ~ f(\xx_{n-1}) = 0  \\
 ~ f(\xx_n) = \infty,~ f(\xx_{n-1}) = \infty, ~\&~ h(\xx_n) \geq h(\xx_{n-1})   \\
 ~ f(\xx_n) = f(\xx_{n-1}) = 0, ~\&~ h(\xx_n) \geq h(\xx_{n-1})
\end{cases} 
\end{equation*}

\section{Numerical Experiments }
\label{s:numer}
In this section, we present numerical evidence to assess the empirical performance of the proposed method. 
Due to its generality, we compare our framework against the variants of the three-operator splitting method (TOSM) \cite[Algorithm~1]{Davis2017threeop}. It may, however, be possible to outperform the computational performance with more specialized methods in specific applications. 

We also present runtime comparison against the state of the art interior point methods. 
Note that \cite{Davis2017threeop} also proposes two schemes with ergodic averaging that feature improved theoretical rate of convergence. 
However, we omitted these variants as they performed worse than the original method in practice. 

A fair comparison between the operator splitting schemes is not an easy task due to the large number of tuning parameters of each method. 
For the ease of comparison and the transparency, we fixed $\lambda_n=1$ for all algorithms. 
This is a natural choice since the convergence rates are shown only for this case in \cite{Davis2017threeop}. 
Unless described otherwise, we used the same step parameter $\gamma$ for all algorithms. 
For IFDR, we used the maximum fixed inertia parameter $\tau_n = \tau$ that satisfies \Cref{thm:IFDR-monotone-inclusion} for the given $\gamma$ and $\lambda_n$. 

\subsection{Markowitz portfolio optimization}

In Markowitz portfolio optimization problem, we set a target return and aim to reduce the risk by minimizing the variance. 
This problem can be formulated as a convex optimization problem as in \cite{Brodie2009}:
$$
\begin{aligned}
& \underset{\xx \in \mathbb{R}^d}{\text{minimize}} & & \E \left[ \vert \boldsymbol{a}_i^T \xx - b \vert^2 \right]  \\
& \text{subject to} & & \xx \in \Delta, \quad \boldsymbol{a}_{av}^T~\xx \geq b
\end{aligned}
$$
where $\Delta$ is the standard simplex,
$\boldsymbol{a}_{av}=\E \left[ \boldsymbol{a}_i \right] $ is the mean return of each asset that is assumed to be known, and $b$ denotes the target return. 

We use 4 different real portfolio datasets that are also considered by \cite{yurtsever2016stcm, Borodin2004}: 
Dow Jones industrial average (DJIA, $30$ stocks, $507$ days), 
New York stock exchange (NYSE, $36$ stocks, $5651$), 
Standard \& Poor's 500 (SP500, $25$ stocks, $1276$ days) and 
Toronto stock exchange (TSE, $88$ stocks, $1258$ days) 

We replicate the experimental setup considered in \cite{yurtsever2016stcm}: 
We split all datasets into test ($10 \%$) and train ($90 \%$) partitions uniformly random. 
We set the desired return as the average return over all assets in the training set, $b = \mathrm{mean}(\boldsymbol{a}_{av})$, and 
we start all algorithms from the zero vector.
We first roughly tuned TOSM and found the best step size parameter as $\gamma = 1.99/L$. 
For this choice, \Cref{thm:IFDR-monotone-inclusion} enforces $\tau_n = 0$ (in which case IFDR is equivalent to TOSM). 
Nevertheless, IFDR-R outperforms its competitors by adapting to the best fixed inertia parameter. 
The results of this experiment are compiled in \Cref{fig:POfig}. 
We compute the objective function over the datapoints in the test partition, $h_{\mathrm{test}}$. 

\subsection{Matrix completion} 		  
\label{sec:matcomp}

We present the results for solving the matrix completion problem with MovieLens 100K benchmark, which consists of $100,\!000$ ratings $b \in \{1,2,3,4,5 \}$ from $1000$ users on $1700$ movies. 
Let $E$ be the training set, and define the associated sampling operator $A: \mathbb{R}^{n \times p} \to \mathbb{R}^d$. 
Then, we can formulate this problem as follows:
\begin{equation*}
\begin{aligned}
& \underset{X \in \mathbb{R}^{m \times p}}{\text{minimize}} & &  \tfrac{1}{2} \| AX - b \|^2 + \rho \|X\|_\ast  \\
& \text{subject to} & & 1 \leq X_{ij} \leq 5
\end{aligned}
\end{equation*}
where $\| \cdot \|_{*}$ denotes the nuclear norm (i.e., sum of the singular values). 

We use the default \texttt{ub} test and train partitions of the data. We remove the movies that are not rated by any user, and the users that have not rated any movie. 
We chose $\rho = 8.4$ via cross validation. 

We tried few different $\gamma$. TOSM performs best when $\gamma = 1.99/L$. 
For this choice, IFDR is equivalent to TOSM, and IFDR-R performs almost the same. 
However, we observed a notable performance improvement for IFDR-R with smaller $\gamma$. 
IFDR with maximum $\tau$ satisfying the condition fails to impress, yet we observed that IFDR can be tuned to get a similar performance as IFDR-R. 
See Figure~\ref{fig:RMSEMC}.

\begin{table*}[ht]
\small
\begin{center}
\begin{tabular}{c|c||ccccccccc}
{Data set} 	& $d$ 	& {\bf IFDR} 	& {\bf IFDR-R} 	& {\bf TOSM} 	 & {\bf TOSM-$\mu$} 	 & {\bf Mosek}	&{\bf SDPT3}	&{\bf SeDuMi} \\
\hline 
LF10        		& 18 		& 0.016		& 0.007		& 0.029		& 0.005		& 1.489		& 0.620		& 0.412			\\
karate        	& 34 		& 0.024		& 0.024		& 0.024		& 0.035		& 1.416		& 1.462		& 1.280			\\
will57        	& 57		& 0.032		& 0.035		& 0.029		& 0.029		& 3.260		& 5.234		& 23.629 			\\
dolphins        	& 62		& 0.036		& 0.039		& 0.037		& 0.040		& 4.752		& 6.833		& 31.013			\\
ash85        	& 85		& 0.051		& 0.066		& 0.046		& 0.048		& 16.072		& 23.785		& 393.584			\\
football        	& 115	& 0.082		& 0.108		& 0.098		& 0.108		& 61.962		& 104.238		& -				\\
west0156        	& 156	& 0.066		& 0.042		& 0.089		& 0.021		& 290.025		& -		& -					\\
jazz        		& 198	& 0.462		& 1.075		& 0.604		& 0.482		& 1759.638	& -		& -					\\
\end{tabular}
 \vspace{-.5em}
\end{center}
\caption{CPU time for projecting a $d \times d$ random matrix onto DNN cone in seconds. Experiments are done in Matlab R2015a on a MacBook with 2.6 GHz quad-core Intel Core i7 CPU with 6MB shared L3 cache and 16 GB 1600 MHz DDR3 memory. Data sets are chosen from \cite{SparseMatrix}.}
 \label{tab:DNNexperiment}
\end{table*}

 \begin{figure*}[ht!]
 \begin{center}
\includegraphics[width=\linewidth]{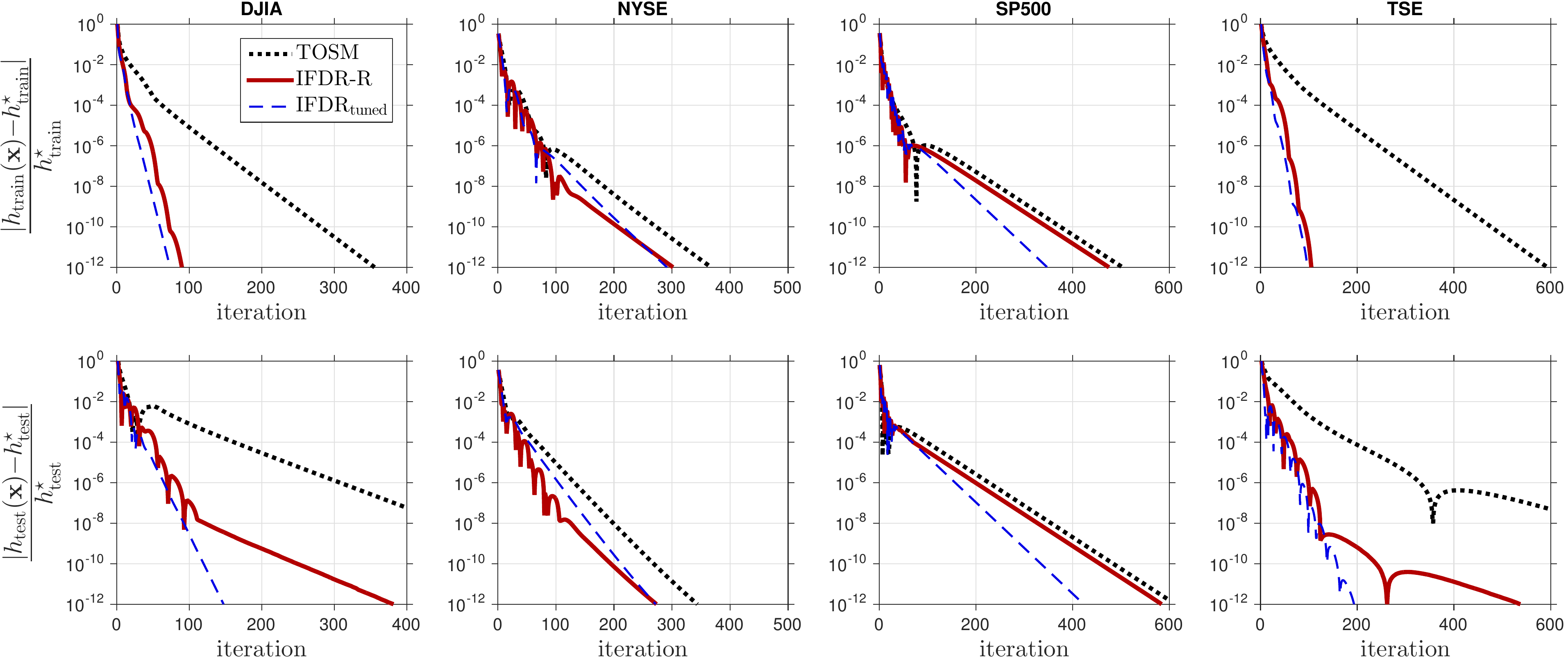} 
\vspace{-.6em}
\caption{{\bf  Portfolio opt.} Columns represent different datasets. Loss on the train {\em(Top)} and test {\em(Bottom)} data. }
\label{fig:POfig}
\end{center}
\end{figure*} 

\begin{figure*}[ht!]
\begin{center}
\includegraphics[width=\linewidth]{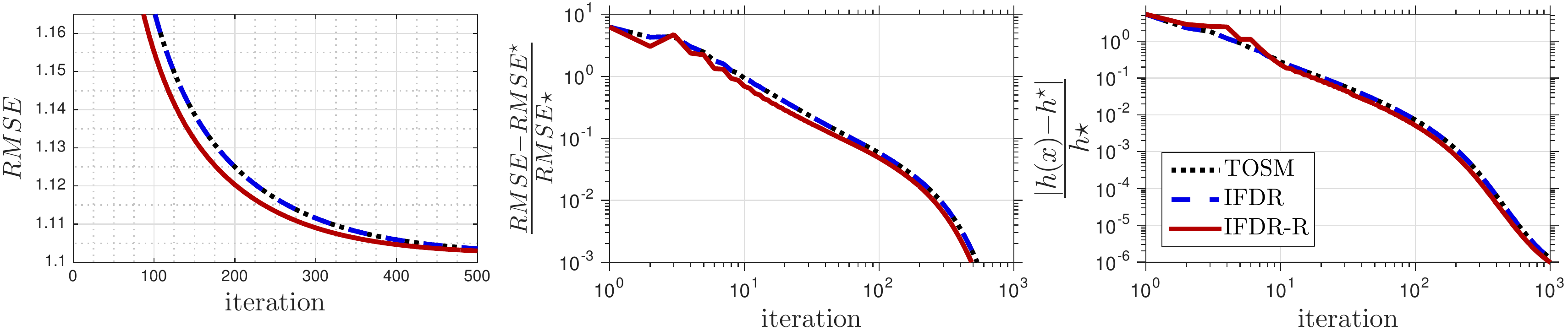} 
\includegraphics[width=\linewidth]{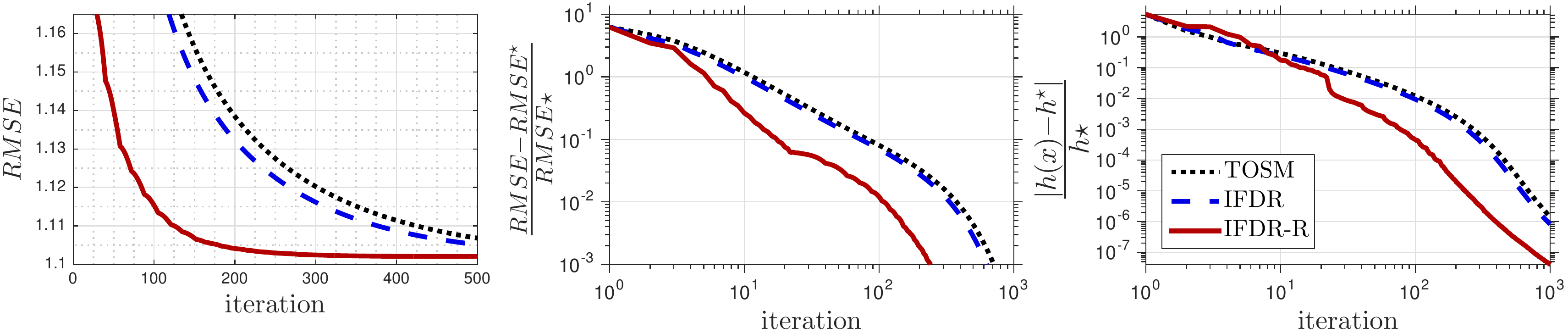} 
\includegraphics[width=\linewidth]{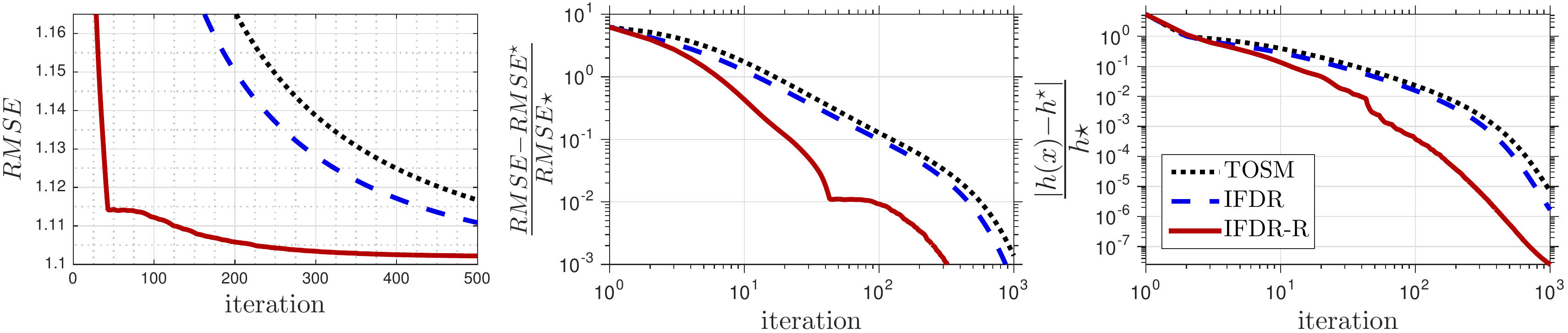} 
\includegraphics[width=\linewidth]{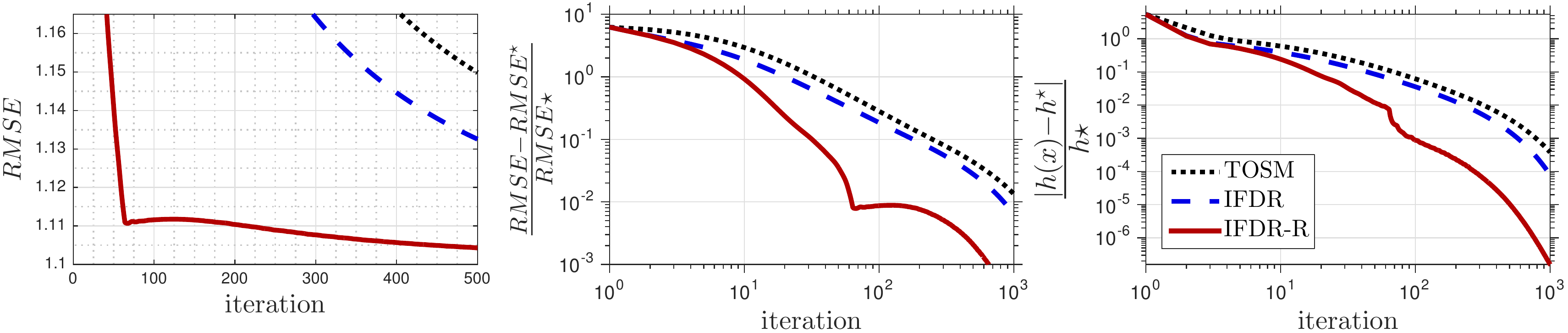} 
\vspace{-.6em}
\caption{ {\bf Matrix completion.} {Convergence behavior of IFDR compared to TOSM in matrix completion problem with various choices of $\gamma$.  
($\tfrac{1.99}{L}$, $\tfrac{1.5}{L}$, $\tfrac{1}{L}$ and $\tfrac{0.5}{L}$ from top to bottom.)} }
\label{fig:RMSEMC}
\end{center}
\end{figure*}

 \begin{figure*}[ht!]
 \begin{center}
\includegraphics[width=\linewidth]{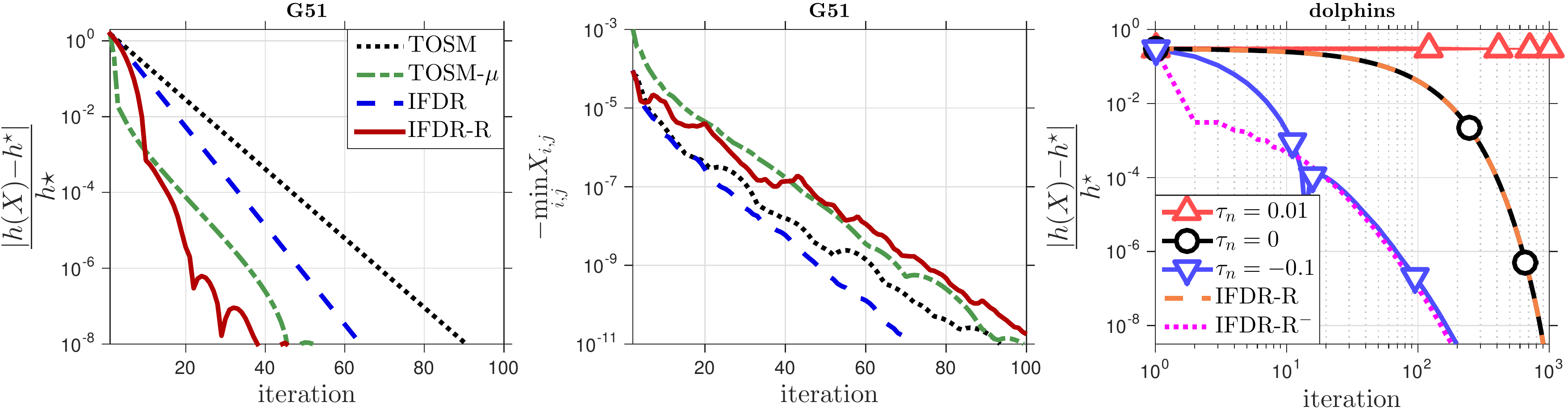} 
\vspace{-.6em}
\caption{{\bf Projection onto DNN cone.} {\em(Left \& Middle)} TOSM(-$\mu$) and IFDR(-R) on `G51' dataset ($d = 10^3$). 100 iterations of IFDR takes 17.262 sec. {\em(Right)} IFDR with different $\tau$ on `dolphins' dataset. }
\label{fig:DNNProj}
\end{center}
\end{figure*} 

\clearpage

\subsection{Projections to doubly nonnegative cone }
\label{sec:DNNproj}

A positive semidefinite matrix with nonnegative coefficients is said to be doubly nonnegative (DNN). 
Optimization over the cone of DNN matrices is effective for an important class of NP-hard optimization problems, but these problems are computationally challenging due to the complexity of the DNN cone. 

We consider the projection of a matrix onto DNN cone:
\begin{equation} \label{eq:DNNproblem}
\begin{aligned}
& \underset{X \in \mathbb{R}^{d\times d}}{\text{minimize}} & &  \tfrac{1}{2} \| X - Z \|_F^2\\
& \text{subject to} & & X \geq 0, \quad X \succeq 0,
\end{aligned}
\end{equation}
where $\|\cdot\|_F$ denotes the Frobenius norm. 
We compare our framework against TOSM and its variant for strongly convex objectives TOSM-$\mu$ \cite[Algorithm~2]{Davis2017threeop}.
We also compare against the state of the art interior point methods: SeDuMi, SDPT3 and Mosek under CVX framework \cite{sedumi,sdpt3,mosek,cvx}. 

We use datasets from \cite{SparseMatrix}, and configure the test setup as follows: 
First, we solve the problem with the aforementioned CVX solvers. 
Then, we run operator splitting methods until they satisfy the following stopping criteria:
\begin{enumerate}
\item $\displaystyle{\min_{i,j} X_{i,j} \geq 0 ~~\text{or}~~  \min_{i,j} X_{i,j} \geq \min_{i,j} X^{\mathrm{Mosek}}_{i,j}}$ 
\vspace{-2mm}
\item $ \displaystyle{\| X - Z \|_F \leq \| X^\mathrm{Mosek} - Z \|_F}$
\end{enumerate}
We set our stopping criteria with respect to Mosek since it is more scalable than the other two. 
Note that the iterates are exactly positive semidefinite, since $\xx_n$ is obtained by projecting $\ww_n$ onto this cone. 
$\gamma = 1.99/L$ did not perform well in this example in contrast with the previous experiments. 
Some rough tuning yields $\gamma = 0.1$ that worked well with all datasets both for TOSM and IFDR(-R).
Note that TOSM$-\mu$ has a dynamic step size $\gamma_n$, which also requires a tuning parameter $\eta$. 
We tuned it as $\eta = 0.1$. 
We initialized all methods from the zero matrix. 
\Cref{tab:DNNexperiment} presents the CPU time of different methods. 

We also ran an instance with $10^3 \times 10^3$ dimensional matrix. 
In this case, we approximated $h^\star$ by $10^4$ iterations of TOSM$-\mu$. 
Results of this experiments are shown on the left and middle panels of \Cref{fig:DNNProj}.

As a final remark, we underline that a small $\tau_n = \tau > 0$ caused IFDR to fail when $\gamma = 1.99/L$ and $\lambda_n=1$. 
This empirically proves the tightness of the conditions listed in \Cref{thm:IFDR-monotone-inclusion}, which enforces $\tau_n = 0$ for these choices. 
Remark that IFDR-R works well even in this difficult setting. 
In view of \Cref{rem:negativetau}, we also tried tuning $\tau_n$ with negative values. 
We observed that a modified version of IFDR-R, which uses the negative of $\tau_n$ can adapt to the best negative parameter. 
These results are compiled in the right panel of \Cref{fig:DNNProj}.

\subsection{Arbitrarily slow example of TOSM}
\label{sec:ArbitrarySlowj}

Intriguingly, we can show that IFDR guarantees $\mathcal{O}(1/n^2)$ convergence rate for solving the pathological example presented in \cite[Section~3.4]{Davis2017threeop}. 
In this section, we first briefly describe this example, we prove the rate and present the numerical demonstrations. 

Let us consider $\HH = \RR^2\oplus\RR^2\ldots$, and 
$(\zeta_n)_{n\in\NN}$ be a sequence in $\left]0, \pi/2\right]$ such that $\zeta_n\to 0$, set $e_0 =[1,0]\in\RR^2$ and $e_{\zeta} = R_{\zeta}e_0$, where 
$R_{\zeta}$ is the counterclockwise rotation in $\RR^2$ by $\zeta$ degrees. 
Define the closed vector subspaces $V$ and $V_1$ as follows:
\begin{equation}
V = \RR^2e_0\oplus\RR^2e_0\ldots 
\quad \text{and}\quad
V_1= \RR^2e_{\zeta_0}\oplus\RR^2e_{\zeta_1}\ldots.
\end{equation}

The problem is to minimize the sum $f+g+h$, where these terms are defined as follows:
\begin{equation}
g = \iota_{V} \quad\quad\quad f = \iota_{V_1} + (\rho/2)\|\cdot\|^2 \quad\quad\quad h = \frac{1}{2}\|\cdot\|^2.
\end{equation}
Here, $\iota$ denotes the indicator function.

It is shown in \cite[Theorem~3.4]{Davis2017threeop}, that for TOSM (recall that IFDR recovers TOSM as a special case with $\tau_n \equiv 0$) the sequence $(\overline{\xx}_n)_{n\in\NN}$ converges arbitrarily slow to $0$ even if $(\yy_n)_{\in\NN}$ converges to $0$ with the rate $o(1/\sqrt{n})$. 

Next, we prove that IFDR with the proper choice of sequence $(\tau_n)_{n \in \NN}$ converges with a guaranteed convergence rate in this example. 

\begin{lemma} \label{t:2} 
Assume that $g = \iota_C$ for some closed convex set $C$. 
Let $\theta_n =\mathcal{O}( 1/n^{s})$ for some $s\in \left]0,1\right]$, and choose 
\begin{equation} \label{eqn:condsThm34}
\gamma \leq \frac{1}{L}, \quad\quad \lambda_n \equiv 1 \quad\quad \text{and} \quad\quad \tau_n = \frac{\theta_n(1-\theta_{n-1})}{\theta_{n-1}}.
\end{equation}
Suppose that 
$$
(1-\theta_n)\ee_n  = \mathcal{O}(\theta_{n}^2),
$$
where 
$$
\ee_n =\max\{0, \scal{\xx_n-\yy_n}{\ww_{n-1}-\xx_{n-1}} + \scal{\yy_{n-1} -\xx_n}{\ww_n-\xx_n} \} .
$$
Define $\zz_n$ as
$$
\zz_{n} = \overline{\xx}_n + \frac{1-\theta_{n-1}}{\theta_{n-1}}(\overline{\xx}_n-\overline{\xx}_{n-1}).
$$
Assume that
$$
\|\zz_n-\xx^\star\|^2 - \|\zz_{n+1}-\xx^\star\|^2  
$$ 
is bounded for all $n \in \NN_+$.  

Then, the following estimate holds:
$$
(h+f)(\yy_n) - (h+f)(\xx^\star) = \mathcal{O}(\theta_{n}) = \mathcal{O}(1/n^s).
$$
 \end{lemma}
 
 \begin{proof}
     We have $\scal{\xx^\star-\xx_n}{\ww_n-\xx_n} \leq 0$ since $g=\iota_C$. Hence
      \begin{alignat}{2}
     \xi_n &\leq \frac{1-\theta_n}{\gamma} ( \scal{\xx_n-\yy_n}{\ww_{n-1}-\xx_{n-1}} + \scal{\yy_{n-1} -\xx_n}{\ww_n-\xx_n} ) \notag\\
     &\leq \frac{1-\theta_n}{\gamma}\max\{  \scal{\xx_n-\yy_n}{\ww_{n-1}-\xx_{n-1}} + \scal{\yy_{n-1} -\xx_n}{\ww_n-\xx_n},0\}\notag\\
     &=\frac{1-\theta_n}{\gamma}\ee_n 
              \end{alignat}
           Under the conditions listed in \Cref{t:2}, we have the following bound: 
           \begin{equation*}
           F(\yy_n) - F(\xx^\star) \leq (1-\theta_n)  F(\yy_{n-1}) - F(\xx^\star) + \mathcal{O}(\theta^2).
           \end{equation*}
           This implies 
           \begin{equation*}
           F(\yy_n) - F(\xx^\star) = \mathcal{O}(\theta_n) = \mathcal{O}(1/n^s).
           \end{equation*} 
           \end{proof}

\begin{remark}
In the case when $C=\RR^d$, IFDR reduces the well-known inertial forward-backward algorithm investigated in \cite{Villa2013, attouch2016,attouch2017,Beck2009fista,Chambolle2015,Bonettini2017}. 
Furthermore, in this special case the condition \eqref{e:sum1} is satisfied since $\overline{\xx}_{n} = \yy_{n-1}$.  
We also empirically verified that the rate $\mathcal{O}(1/n^2)$ does not hold may not hold when \eqref{e:sum1} is not satisfied.
\end{remark}

\begin{figure}[h]
 \begin{center}
\includegraphics[width=.85\linewidth]{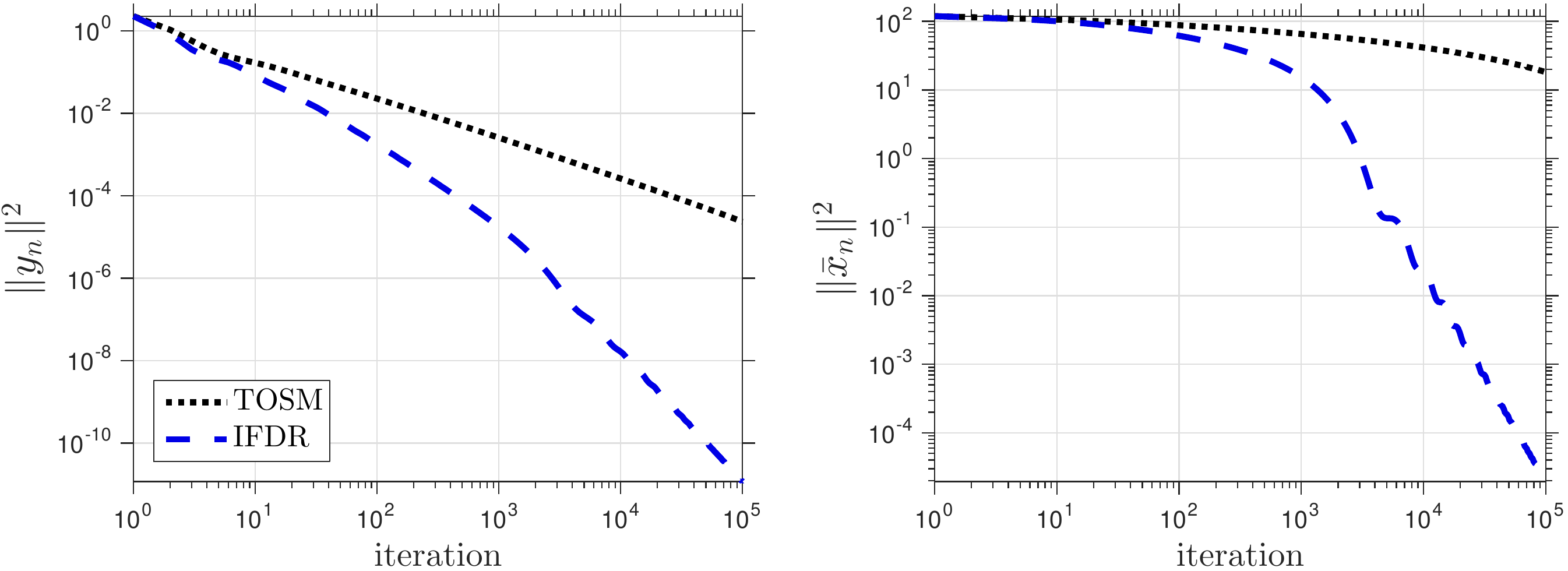} 
\caption{Convergence behavior of IFDR and TOSM for the concept problem in \cite[Section~3.4]{Davis2017threeop}. 
In contrast to TOSM, IFDR does not converge arbitrarily slow in $\|\overline{\xx}_n\|$, indeed it attains empirical $\mathcal{O}(1/n^2)$ rate.}
\label{fig:abslow}
\end{center}
\end{figure} 

\begin{theorem} \label{t:3}
Assume that $g = \iota_C$ for some closed convex set $C$. 
Let us choose $\gamma$, $(\lambda_n)_{n\in\NN_+}$ and $(\tau_n)_{n\in\NN_+}$ as in \eqref{eqn:condsThm34}, with $\theta_n =t/(n+t)$ for some $t\geq 2$. Suppose that
\begin{equation}
\label{e:sum1}
\sum_{n\in\NN_{+}} \theta_{n}^{-2}\ee_n < +\infty,
\end{equation}
where 
\begin{equation*}
\quad \ee_n =\max\{0, \scal{\xx_n-\yy_n}{\ww_{n-1}-\xx_{n-1}} + \scal{\yy_{n-1} -\xx_n}{\ww_n-\xx_n} \} .
\end{equation*}
Then, the following holds:
\begin{equation}
 \label{e:re1}
(h+f)(\yy_n) - (h+f)(\xx^\star) = \mathcal{O}(\theta_{n}^2) = \mathcal{O}(1/n^2).
\end{equation}
Furthermore, we have:
\begin{equation*}
\sum_{n\geq 1} n\left((f+h)(\yy_{n-1}) - (f+h)(\xx^\star)\right) < \infty. 
\end{equation*}
\end{theorem}

Intriguingly, the condition \eqref{e:sum1} is satisfied in this particular ``worst case'' example. Indeed,  we have $\nabla h(\xx_n) = \xx_n$, and hence, it follows that 
    \begin{equation*}
    \xx_n - \ww_n - \yy_n \in V_{1}^{\perp}.
    \end{equation*}
    Since $\yy_{n-1}\in V_1$, we have 
    $$\scal{\yy_{n-1}-\xx_n}{\xx_n-\ww_n} = \scal{\yy_n}{\yy_{n-1}}.$$ 
    By the same way, we also have 
    $$\scal{\yy_{n}-\xx_{n}}{\xx_{n-1}-\ww_{n-1}} = \scal{\yy_n}{\yy_{n-1}}.$$ 
    Subtracting these equalities, we see that our condition holds for the fast convergence. 
    These results are also numerically illustrated in Figure \ref{fig:abslow}.

\begin{proof}
 Set $F = f+h$ and  $f_n = f + \frac{1}{\gamma}\scal{\cdot}{\ww_n-\xx_n} $. Then
 \begin{alignat}{2}
 \quad &\yy_n = \prox_{\gamma f}(2\xx_n-\ww_n - \gamma\nabla h(\xx_n)) \notag\\ 
 \quad\Leftrightarrow\; &2\xx_n-\ww_n - \gamma\nabla h(\xx_n) -\yy_n \in \gamma \partial f (\yy_n)\notag\\
  \quad\Leftrightarrow\; &\xx_n- \gamma\nabla h(\xx_n) -\yy_n \in \gamma \partial f(\yy_n)+\ww_n-\xx_n \notag\\
    \quad\Leftrightarrow\; &\xx_n- \gamma\nabla h(\xx_n) -\yy_n \in \gamma\partial( f+ \tfrac{1}{\gamma}\scal{\cdot}{\ww_n-\xx_n})(y_n) \notag\\
       \quad\Leftrightarrow\; & \yy_n = \prox_{\gamma f_n}(\xx_n-\gamma \nabla h(\xx_n)). \notag
 \end{alignat}  
 Set $F_n = f_n + h$ and $\yy^\star = (1-\theta_n)\yy_{n-1} + \theta_n \xx^\star.$
 Then, it follows from \cite[Lemma 8]{Bonettini2017} that
 \begin{alignat*}{2}
  F_n(\yy_n) + \frac{1}{2\gamma}\|\yy^\star-\yy_n \|^2\leq F_n(\yy^\star) + \frac{1}{2\gamma}\|\yy^\star-
 \xx_n \|^2 .
  \end{alignat*}
 Hence 
 \begin{alignat*}{2}
 F(\yy_n) + \frac{1}{\gamma}\scal{\yy_n-\yy^\star}{\ww_n-\xx_n}+ \frac{1}{2\gamma}\|\yy^\star-\yy_n \|^2 \leq F(\yy^\star) +  \frac{1}{2\gamma}\|\yy^\star-
 \xx_n \|^2
  \end{alignat*}
 or
  \begin{alignat*}{2}
 F(\yy_n) + \frac{1}{2\gamma}\|\overline{\xx}_{n+1}-\yy^\star\|^2\leq F(\yy^\star) +  \frac{1}{2\gamma}\|\yy^\star-
 \xx_n \|^2+ \frac{1}{2\gamma}\|\ww_n-\xx_n \|^2.
  \end{alignat*}
    Now, using the convexity of $h$, we have 
  \begin{alignat*}{2}
  F(\yy_n) - F(\xx^\star) + & \frac{1}{2\gamma}\|\overline{\xx}_{n+1}-\yy^\star\|^2 
 & \leq (1-\theta_n)[ F(\yy_{n-1}) - F(\xx^\star) ] +  \frac{1}{2\gamma}\|\yy^\star-
 \xx_n \|^2+ \frac{1}{2\gamma}\|\ww_n-\xx_n \|^2.
    \end{alignat*}
     We have 
     \begin{alignat*}{2}
     \frac{1}{2\gamma}\|\yy^\star-
 \xx_n \|^2+ \frac{1}{2\gamma}\|\ww_n-\xx_n \|^2& = \frac{1}{2\gamma} \| \ww_n-\yy^\star\|^2+ \frac{1}{\gamma}\scal{\yy^\star-\xx_n}{\ww_n-\xx_n}.\\
   \end{alignat*} 
    Set 
    \begin{alignat*}{2}
    \zz_{n} = \overline{\xx}_n + \frac{1-\theta_{n-1}}{\theta_{n-1}}(\overline{\xx}_n-\overline{\xx}_{n-1}).
    \end{alignat*}
    Then 
     \begin{alignat}{2}
     \label{e:tet1}
  F(\yy_n) - F(\xx^\star) +  \frac{\theta_{n}^2}{2\gamma}\|\zz_{n+1}-\xx^\star\|^2\leq (1-\theta_n)[ F(\yy_{n-1}) - F(\xx^\star) ] + \frac{\theta_{n}^2}{2\gamma}\|\zz_n-\xx^\star \|^2 +\xi_n,
    \end{alignat}
    where 
    \begin{equation*}
    \xi_n= \theta_n\frac{(1-\theta_n)}{\gamma}\scal{\zz_n-\zz_{n+1}}{\ww_{n-1}-\xx_{n-1}} + \frac{1}{\gamma}\scal{\yy^\star-\xx_n}{\ww_n-\xx_n} .
      \end{equation*}
Now, we derive that the sequence 
$(\theta_{n}^{-2} (F(\yy_n) - F(\xx^\star)) +  \frac{1}{2\gamma}\|\zz_{n+1}-\xx^\star\|^2)_{n\in\NN_+}$ is bounded. Hence,
\begin{equation*}
F(\yy_n) - F(\xx^\star)=\mathcal{O}(\theta_{n}^2) = \mathcal{O}(1/n^2),
\end{equation*}
which proves the desired result \eqref{e:re1}.  Let us set $t_n = 1/\theta_n$ and 
$s_n =F(\yy_n) - F(\xx^\star) $. Then, it follows from \eqref{e:tet1} that 
\begin{equation*}
t^{2}_n s_n -t^{2}_{n-1}s_{n-1}+  (t^{2}_{n-1}-t^{2}_n+t_n) s_{n-1} \leq \frac{1}{2\gamma}\|\zz_n-\xx^\star \|^2- \frac{1}{2\gamma}\|\zz_{n+1}-\xx^\star\|^2 + t^{2}_n\xi_n.
\end{equation*}
Summing from $n=2$ to $n= N$ we get 
\begin{alignat*}{2}
t^{2}_N s_N -t^{2}_{1}s_{1} +\sum_{n=2}^N (t^{2}_{n-1}-t^{2}_n+t_n) s_{n-1} \leq \frac{1}{2\gamma}\|\zz_2-\xx^\star \|^2- \frac{1}{2\gamma}\|\zz_{N+1}-\xx^\star\|^2
+\sum_{n=2}^Nt^{2}_n\xi_n\notag\\
\leq \frac{1}{2\gamma}\|\zz_2-\xx^\star \|^2
+\sum_{n\in\NN_+}t^{2}_n\xi_n,
\end{alignat*}
which implies that $\sum_{n=2}^N (t^{2}_{n-1}-t^{2}_n+t_n) s_{n-1} < +\infty$. Since $t^{2}_{n-1}-t^{2}_n+t_n \geq n(a-2)/a^2$, we get $(ns_{n-1})_{n\geq 2}$ is summable. 
\end{proof}

\subsubsection*{Acknowledgements}

This work was supported by the European Commission under Grant ERC Future Proof. 

\subsection*{References}
\bibliographystyle{ieeetr}
\bibliography{references.bib}

\end{document}